\documentclass[12pt]{article}

\def\edo{\end{document}}

\usepackage{amsmath, amsthm, amscd, amsfonts, amssymb}

\newtheorem{theorem}{Theorem}[section]

\newtheorem{corollary}[theorem]{Corollary}
\newtheorem{lemma}[theorem]{Lemma}

\newtheorem{remark}[theorem]{Remark}

\def\divv{{\rm div }}
\def\rrd{{\mathbb{R}^d}}

\def\calf{{\mathcal{F}}}

\def\calo{{\mathcal{O}}}

\def\calb{{\mathcal{B}}}
\def\cald{{\mathcal{D}}}
\def\calx{{\mathcal{X}}}
\def\calk{{\mathcal{K}}}
\def\call{{\mathcal{L}}}

\def\calp{{\mathcal{P}}}

\def\vsp{\vspace*{1,5mm}\\ }
\def\bk{\bigskip }
\def\mk{\medskip }

\def\n{\noindent }
\def\dd{\displaystyle}

\def\barr{\begin{array}}
\def\earr{\end{array}}

\def\bit{\begin{itemize}}
\def\eit{\end{itemize}}

\def\FP{Fokker--Planck}

\def\1{^{-1}}

\def\nn{{\mathbb{N}}}

\def\rr{{\mathbb{R}}}
\def\9{{\infty}}
\def\lbb{{\lambda}}
\def\wt{\widetilde}
\def\ov{\overline}
\def\vf{{\varphi}}
\def\oo{{\omega}}
\def\ooo{{\Omega}}
\def\pp{{\partial}}
\def\vp{{\varepsilon}}

\def\ff{\forall }
\def\({\left(}
\def\){\right)}
\def\<{\left<}
\def\>{\right>}

\def\fpe{Fokker--Planck equation}
\def\fpes{Fokker--Planck equations}
\def\mkv{McKean--Vlasov}
\def\mkve{McKean--Vlasov equation}

\title{Nonlinear \fpes\ with~singular integral drifts and~McKean--Vlasov SDEs} 
\author{Viorel Barbu\thanks{Octav Mayer Institute of Mathematics of  Romanian Academy,  Ia\c si, Romania.  Email:~vb41@uaic.ro}} 
\date{}

\begin{document}
\maketitle

\begin{abstract}
\n In this work, one proves the well-posedness in the Sobolev space $H\1(\rrd)$ of the Cauchy problem $$\barr{r}
\dd\frac{\pp u}{\pp t}-\Delta\beta(u)+\divv((D(x)b(u)+K*u)u)=0\\
\mbox{ in $(0,\9)\times\rrd;$ $u(0,x)=u_0(x),$}\earr$$ where $d\ge2$, $\beta$ is a con\-ti\-nuous monotonically increasing function, \mbox{$D:\rrd\to\rrd$, $b:\rr\to\rr^+$} are appropriate functions and \mbox{$K\!\in\! C^1(\rrd\!\setminus\!0)$} is a singular kernel. One proves also the uniqueness of distributional solutions $u\in L^1\cap L^\9$ which are weakly (narrowly) continuous in $t$ from $[0,\9)$ to $L^1(\rrd)$ and  the uniqueness for the corresponding li\-nearized equations, As a consequence,  it follows the existence and uniqueness of strong solutions to the \mkv\ stochastic differential equation (SDE)
$$\begin{array}{l}
	dX_t=(D(X)b(u(t,X_t)+(K*u(t,\cdot))(X_t))dt +\(\dd\frac{2\beta(u(t,X_t))}{u(t,X_t)}\)dW_t\\
	\call(X_t)=u(t,\cdot),\ X(0)=X_0,\ u_0=\mathbb{P}\circ X^{-1}_0,
	\end{array}$$
where $\call(X_t)$ is the density of the probability law of the process $X_t$ with respect to the Lebesgue measure on $\rrd$. Moreover, the laws of solutions to this equation have the Markov property.\mk
\end{abstract}\vfill

\section{Introduction}\label{s1}
Here we are concerned with the nonlinear \fpe\ (abbre\-viated NFPE)

\begin{equation}\label{e1.1}
\barr{ll}
\dd\frac{\pp u}{\pp t}-\Delta\beta(u(t,x))+{\rm div}((D(x)b(u(t,x))\\
\qquad+(K*u(t,\cdot))(x)u(t,x))=0\mbox{ in }(0,\9)\times\rrd,\\
u(0,x)=u_0(x),\ x\in\rrd,\vspace*{-2mm} \earr 
\end{equation}
where $d\ge2$, $\beta:\rr\to\rr,\ b:\rr\to\rr^+,\ D:\rrd\to\rrd$, $K:\rrd\to\rrd$ are given functions to be made precise below and the symbol $*$ stands for the convolution product in the space $L^1(\rrd)$.

The following hypotheses will be in effect throughout this work.
\begin{itemize}
	\item[\rm(i)] $\beta\in C^1(\rr)$, $\beta(0)=0,\ \beta'(r)\ge\alpha>0,$ $\ff r\in\rr.$ 
	\item[\rm(ii)]  $D\in L^\9(\rrd;\rrd)\cap L^2(\rrd;\rrd), 
	 (\divv\ D)^-=0$. 
	 %\in L^\9(\rrd)$.  
	\item[\rm(iii)] $b\in C^1(\rr),\ b(r)\ge0,\ \ff r\in\rr.$  	
	\item[\rm(iv)] $K\in C^1(\rrd\setminus0;\rrd)\cap L^1(B_1;\rrd)\cap L^2(B^c_1;\rrd)$ for some $1\le p\le2$ and  
	\begin{equation}\label{e1.2}
	\hspace*{-3,4mm}(\divv\, K)^-\in L^\9(\rrd),\, (K(x){\cdot} x)^-|x|\1\in L^\9(B_1),
	\end{equation}
where $B_r=\{x\in \rrd;\ |x|\le r\},\ B^c_r=\{x\in\rrd;\ |x|>r\},\ \ff r>0.$		
\end{itemize}
As is well known, NFPE \eqref{e1.1} is related to the \mkv\ SDE
\begin{equation}\barr{rcl}
	dX_t&=&(D(X_t)b(u(t,X_t))+(K*u(t,\cdot))(X_t))dt\label{e1.3}\vsp
	&&+\(\dd\frac{2\beta(u(t,X_t))}{u(t,X_t)}\)^{\frac12}dW_t,\ t\ge0,\earr\end{equation}in a probability space 
	$(\ooo,\calf,\mathbb{P})$ with normal filtration $(\calf_t)_{t>0}$ and an $(\calf_t)$-
	Brownian motion $W_t)$ through the stochastic representation
	\begin{equation}
	\call_{X_t}(x)=u(t,x),\ \ff t\ge0;\ u_0(x)=\mathbb{P}\circ X^{-1}_0,\ x\in\rrd.\label{e1.4}
\end{equation}Here,  $\call_{X_t}$ 
is the density of the marginal law $\mathbb{P}\circ X^{-1}_t$ of $X_t$  with respect to the Lebesgue measure.

In the special case $K\equiv0$, i.e. for   Niemitsky type drift, NFPE \eqref{e1.1} was studied in the works \cite{5}--\cite{7} (see, also, the book \cite{8}), where the existence of mild solutions in $L^1(\rrd)$ via nonlinear semigroup theory and uniqueness of distributional solutions was obtained under hypotheses (i)--(iii). As a  consequence, the existence and uniqueness of a strong solution to \eqref{e1.3} was derived via the superposition principle and the Yamata--Watanabe theorem. The main novelty of this work is the presence of $(K*u)u$ in the drift term of NFPE \eqref{e1.1} which from the point of statistical mechanics this term in \fpe\ models the long-range particle interactions (see \cite{10}--\cite{12}), as well as chemostatic biological processes (\cite{10}, \cite{16}, \cite{17}, \cite{19}). These equations can be seen as a part of a more general class of nonlinear \FP\ equations with irregular drift tgerms in the spatial variable. (See, e.g., the work \cite{13a} for recent results and references.) Below there are a few examples of such kernels $K$. 
 
 \bk\n{$\bf 1^\circ$.}  
 {\bf Riesz type kernels.} 
The kernel
\begin{equation}\label{e1.6}
	K(x)=\mu x|x|^{s-d-2},\ 0<s<d,\ \mu>0,\ d>2, 
\end{equation}
is derived from the Riesz potential by $K(x)\equiv\nabla I_s(x)$, where
$$ I_s(x)=-\dd\frac\mu{d-s}\ |x|^{s-d}\mbox{ if }0<s<d,\ d>2.$$
We note that, for $0<s<d$ and
$$\mu=(d-s)\pi^{\frac d2}\,2^s\,\Gamma\(\frac s2\)\Big/\Gamma\(\frac{d-s}2\),$$
we have (see, e.g., \cite{22a}, p.~117)
\begin{equation}\label{e1.7}
	(-\Delta)^{-\frac s2}f= I_s*f,\ \ff f\in L^1\cap L^2.
\end{equation}
As easily seen, hypothesis (iv) holds in this case if $0<s<\frac{d+4}2,$ $d>2$.

\bk\n{$\bf 2^\circ$.} {\bf Bessel kernels} 
\begin{equation}\label{e1.8}
	K(x)=\nabla G_\alpha(x), \ x\in\rrd\setminus\{0\},
\end{equation}
where $G_\alpha$ is the Bessel potential defined by (see \cite{1}, \cite{22a}, p. 132)
$$\calf(G_\alpha)(\xi)=-(1+4\pi^2|\xi|)^{-\frac\alpha2},\ \xi\in\rrd.$$
(Here, $\calf(G_\alpha)$ is the Fourier transform of $G_\alpha$.) In analogy with \eqref{e1.7}, we have
  
\begin{equation}\label{e1.9}
	(I-\Delta)^{-\frac \alpha2}f=G_\alpha*f,\ \nabla((I-\Delta)^{-\frac\alpha2}f)=K*f, 
\end{equation}
and
$$\barr{lcl}
G_\alpha(x)&=&-H_{\frac{\alpha-d}2}(|x|)|x|^{\frac{\alpha-d}2},\ 0<\alpha<d+1,\ x\in\rrd,\vsp 
H_\nu(r)&\equiv&\mu_\nu\,r^\nu\,e^{-\nu}\int^\9_0e^{-tr}\,
t^{\nu-\frac12}\(1+\frac12\,t\)^{\nu-\frac12}dt,\ \nu>\-\frac12,\ r>0.\earr$$Then, hypothesis (iv) hold if $0<\alpha<\frac d2$.

NFPE \eqref{e1.1} with kernels $K$ of the form \eqref{e1.6}, 
\eqref{e1.7}  arises   
in the models of chemotaxis phenomena.  Such an example is the generalized Keller--Segal model \cite{16}
\begin{equation}\label{e1.9a}
\barr{lcll}
\dd\frac{\pp u}{\pp t}-\Delta\beta(u)+\divv(u\nabla v)&=&0&\mbox{ in }(0,\9)\times\rr,\vsp
\hfill(-\Delta)^\alpha v&=&u&\mbox{in }(0,\9)\times\rrd,\vsp
\hspace*{-19mm}\mbox{respectively }\vsp 
\hfill(I-\Delta)^\alpha v&=&u&\mbox{where }0<\alpha\le1,\earr 
\end{equation}
which describes the chemotaxis dynamics of biological populations in pre\-sence of anomalous diffusion. In fact, the first equation describes the evolution of density of population and the second the evolution of chemotaxis agent. The same equation can be interpreted as a Cahn--Hilliard equation~\cite{11}. 

We shall prove here the existence and uniqueness of the solution $u:[0,\9)$ $\to H\1(\rr^+)$ to NFPE \eqref{e1.1} in the set $\calp$ of probability densities,  
\begin{equation}\label{e1.10}
\calp=\left\{u\in L^1(\rrd);\ u\ge0,\mbox{ a.e. in }\rrd,\int_\rrd u(x)dx=1\right\}.\end{equation}
Moreover, the flow $t\to u(t)=S(t)u_0$ is a continuous semigroup of quasi-contractions in the Sobolev space $H\1(\rrd)$ which is everywhere differentiable from the right on $[0,\9)$ on a dense subset of $\calp\cap L^1$. The proof is based on K${\rm\bar o}$mura's deep result regarding the differentiability of continuous non\-linear semigroups of contractions defined on a closed convex set of a Hilbert space. As shown in Theorem \ref{t4.1} in Section \ref{s4}, under additional hypothesis this solution is unique in the class of distributional weakly continuous in $t$ solutions to NFPE \eqref{e1.1}. Such a result was proved earlier for NFPE with Niemitsky type kernel \cite{7}, \cite{8} and in \cite{9} for equations with $D\equiv0,$ $\beta(r)\equiv r$ and the Biot--Savart kernel\newpage 
\begin{equation}\label{e1.11b}
K(x)\equiv\frac1\pi\, x^\bot|x|^{-2},\ x=(x_1,x_2),\ x^\bot=(-x_2,x_1),\end{equation}
which is equivalent with the vorticity equation for the $2-D$ Navier--Stokes equation. It should be mentioned, however, that this case is ruled out by hypothesis (iv) which, as seen in \cite{9}, need some specific arguments. 

A direct consequence of the uniqueness result is the existence and uniqueness of a strong solution $X$ to a \mkve\ \eqref{e1.3} (Section \ref{s5}).   

As regards the literature on NFPE \eqref{e1.1} under appropriate assumptions on $K$ and $D\equiv0$, the work \cite{18} of C. Olivera, A Richard and M. Tomasovic should be cited. In this context, the work \cite{17} should be also cited. 
We also note that, if $D=-\nabla\Phi$, $K=-\nabla W$, where $\Phi:\rrd\to\rrd$, $W:\rrd\to\rr$, then the \FP\ equation is associated with the entropy functional (see~\cite{26b})
\begin{equation}
\label{e1.16b}
\barr{ll}
E(u)\!\!\!&\equiv\dd\int_\rrd dx\left(\int^{u(x)}_1
\frac{\beta'(\tau)}{b(\tau)\tau}\,d\tau+\Phi(x)u(x)\right)\\
&+\dd\frac12\int_{\rrd\times\rrd}W(x-y)u(x)u(y)dy.\earr
\end{equation}

\bk \noindent{\bf Notations and definitions.} If $\calo\subset\rrd$ is a Lebesgue measurable set, we denote by $L^p(\calo)$, $1\le p\le\9$, the standard space of $p$-Lebesgue integrable functions on $\calo$ with the norm $\|\cdot\|_{L^p(\calo)}$. If $\calo=\rrd$, we shall simply write $L^p=L^p(\rrd)$ and $|\cdot|_p=\|\cdot\|_{L^p}$. By $C_b(\rrd)$ we shall denote the space of bounded and continuous functions on $\rrd$. Denote by  $H^1$ the Sobolev space $H^1(\rrd)=\left\{u\in L^2;\,\frac{\pp u}{\pp x_i}\in L^2,\,i=1,...,d\right\}$, with the standard norm $|\cdot|_{H^1}$, and by $H\1$ the dual of $H^1$. The space $H\1$ is endowed with the scalar product
$$\<u_1,u_2\>_{-1}=((I-\Delta)\1u_1,u_2)_2,$$where $(\cdot,\cdot)_2$ is the scalar product of the space $L^2$. We denote by $|\cdot|_{-1}$ the corresponding Hilbertian norm of $H\1$, that is,
$$|u|_{-1}=((I-\Delta)\1u,u)^{\frac12}_2,\ \ff u\in H\1.$$
For $0<T\le0$, we shall denote by $C([0,T];H\1$ the space of all $H\1$-valued continuous functions on $[0,T]$. $W^{1,p}([0,T];H\1)$, $1\le p\le\9$, is the space $\left\{u\in L^\9(0,T;H\1);\,\frac{du}{dt}\in L^p(0,T;H\1)\right\}$, where $\frac{du}{dt}$ is taken in the sense of $H\1$-valued vectorial distributions. It turns out (see, e.g., \cite{3}, \cite{4}) that each $u\in W^{1,p}([0,T];H\1)$ is absolutely continuous and a.e. $H\1$ differentiable on $(0,T)$. We denote by $\cald'(Q_T)$ the space of Schwartz distributions on $Q_T=(0,T)\times\rrd.$ 

The nonlinear operator $A:D(A)\subset H\1\to H\1$ is said to be {\it quasi-m-accretive} if there is $\oo>0$ such that, for all $\lbb\in\(0,\frac1\oo\)$,
\begin{equation}
\label{e1.11}
\barr{r}
\left|(I+\lbb A)\1f_1-(I+\lbb A)\1f_2\right|_{-1}\le(1-\lbb\oo)\1|f_1-f_2|_{-1},\vsp \ff f_1,f_2\in H\1.\earr
\end{equation}
Equivalently,
\begin{equation}
\label{e1.12}
\barr{l}
\<Au_1-Au_2,u_1-u_2\>_{-1}\ge-\oo|u_1-u_2|_{-1},\,\ff u_1,u_2\in D(A),\vsp
R(I+\lbb A)=H\1,\ \ff\lbb\in(0,\lbb_0),
\earr 
\end{equation}
for some $\lbb_0\in\(0,\frac1\oo\)$. (Here, $R(I+\lbb A)$ stands for the range of $I+\lbb A$.)

If $A$ is {\it quasi-m-accretive}, then it generates on $\ov{D(A)}$ (the closure of $D(A)$) a continuous semigroup of $\oo$-quasi-contractions $S(t)$, that is (see, e.g., \cite{3},~\cite{4}),
\begin{eqnarray}
&\dd\frac{d^+}{dt}\,S(t)u_0+AS(t)u_0=0,\ \ff t>0,\label{e1.13}\\[1mm]
&\dd\lim_{t\to0}S(t)u_0=u_0,\ S(t+s)u_0=S(t)S(s)u_0,\ \ff t,s\ge0,\label{e1.13b}
\end{eqnarray}
for all $u_0\in D(A)$. Also, the right derivative $\frac{d^+}{dt}\,S(t)u_0$ is continuous from the right on $[0,\9)$ 
and $\frac{d}{dt}\,S(t)u_0$ exists and is everywhere continuous except a countable set of $[0,\9)$. Moreover, for all  $u_0\in\ov{D(A)}$, $S(t)u_0$ is given by the ex\-po\-nen\-tial formula
\begin{equation}
\label{e1.14}
S(t)u_0=\dd\lim_{n\to\9}\(I+\frac tn\,A\)^{-n}u_0\ \ \mbox{ in }H\1
\end{equation}
 uniformly on compact intervals $[0,T]\subset[0,\9)$  and $S(t)$ is $\oo$-quasi-contractive, that is,
\begin{equation}
\label{e1.15}
|S(t)u_0-S(t)v_0|_{-1}\le\exp(\oo t)|u_0-v_0|_{-1},\ 
\ff u_0,v_0\in\ov{D(A)}, \ t\ge0.
\end{equation}
This means that $u(t)=S(t)u_0$ is for each $u_0\in D(A)$ a {\it smooth} ({\it differentiable}) {\it solution} (in the space $H\1$) to the Cauchy problem
\begin{equation}
\label{e1.16}
\frac{du}{dt}+Au=0,\ t\ge0;\ u(0)=u_0,
\end{equation}
while for $u_0\in\ov{D(A)}\setminus D(A),$ the function $u=u(t)$ given by \eqref{e1.14} is a {\it mild solution} to \eqref{e1.16} and it is equivalently defined as the limit for $h\to0$ of the solution $u_h$ to the finite difference scheme
\begin{equation}
\label{e1.18a}\barr{l}
\dd\frac1h\,(u_h(t)-u_h(t-h))+Au_h(t)=0,\ \ff t>0,\vsp
u_h(t)=u_0,\ \ff t<0.\earr\end{equation}

\section{The well-posedness of NFPE \eqref{e1.1}}\label{s2}
\setcounter{equation}{0}

Let $u_0\in L^2$. The function $u:[0,\9)\to L^2$ is said to be a {\it weak solution} to \eqref{e1.1} if
\begin{eqnarray}
&u\in C([0,\9);H\1),\ \dd\frac{du}{dt}\in L^2(0,T;H\1),\ \ff T>0,\label{e2.1}\\[1mm]
&\beta(u)\in L^2(0,T;H^1),\ Db(u)+(K*u)u\in L^2(0,T;L^2),\label{e2.2}\\[2mm]
&\barr{ll}
\raise-3mm\hbox{${}_{H\1}$}\!\!\left(\dd\frac{du}{dt}\,(t),\vf\right)_{\!\!H^1}\!\!\!&+\dd\int_\rrd(\nabla_x\beta(u(t,x))-(Db(u(t,x))\vsp
&+K*u)u(x)){\cdot}\nabla\vf(x)dx=0,\ 
\ff\vf\in H^1,\vsp
&\hfill\mbox{ a.e. }t>0,\mbox{ a.e. }t>0,\earr\label{e2.3}\\[2mm]
&u(0,x)=u_0(x),\ x\in\rrd.\label{e2.4}\end{eqnarray}
We have

\begin{theorem}
	\label{t2.1} Let $u_0\in \calp\cap L^\9$. Then, under hypotheses {\rm(i)--(iv)} there is a unique weak solution $u^*=u(t,u_0)$ to \eqref{e1.1} which satisfies  
	\begin{eqnarray}
	&u^*(t)\in \calp,\ \ff t\ge0,\label{e2.5}\\[1mm]
	&u^*\in L^2(0,T;H^1)\cap L^\9((0,T)\times\rrd),\ \ff T>0.\label{e2.6}
	\end{eqnarray} 
	Moreover, we have
	\begin{eqnarray} 
	&|u(t,u_0)-u(t,v_0)|_{-1}\le\exp(\oo t)|u_0-v_0|_{-1},\ \ff u_0,v_0\in\calp\cap L^\9,\ t\ge0,\ \ \ \ \label{e2.7}\\[1mm]
	&0\le u(t,x)\le\exp(\gamma t)|u_0|_\9,\ \ff(t,x)\in(0,\9)\times\rrd,\ u_0\in\calp\cap L^\9,\ \ \ \label{e2.8}\end{eqnarray}
	for some $\oo,\gamma>0$.
	
	Moreover, the function $u:[0,\9)\to L^1$ is narrowly continuous, that is,
	\begin{equation}
	\label{e2.8a}\lim_{t\to t_0}\int_\rrd u(t,x)\vf(x)dx
	=\int_\rrd u(t_0,x)\vf(x)dx,\ \ff\vf\in C_b(\rrd),\ff t_0\ge0.
	\end{equation}		
\end{theorem}

We set
\begin{equation}\label{e2.9}
S(t)u_0=u(t,u_0),\ t\ge0,\ u_0\in\calp\cap L^\9,\end{equation}
and note that $S(t)(\calp\cap L^\9)\subset\calp\cap L^\9,$ $\ff t\ge0,$ the function $t\to S(t)u_0$ is continuous (in $H\1$) on $[0,\9)$ and by the uniqueness part in Theorem \ref{t2.1} satisfies the semigroup property %(see \eqref{e1.13})
\begin{equation*} 
S(t+s)u_0=S(t)S(s)u_0,\ \ff t,s\ge0,\ u_0\in\calp\cap L^\9,\end{equation*}
and by \eqref{e2.7}
\begin{equation*}
|S(t)u_0-S(t)v_0|_{-1}\le\exp(\oo t)|u_0-v_0|_{-1},\ \ff t\ge0,\ u_,v_0\in\calp\cap L^\9.\end{equation*}

In other words (see \eqref{e1.13b}, \eqref{e1.15}), $S(t)$ {\it is a continuous semigroup of $\oo$-contractions on $\calp\cap L^\9\subset H\1$} which extends by density to a continuous $\oo$-contractive semigroup on the closure $\calk$ of $\calp\cap L^\9$ in $H\1$-norm. Then, by K$\rm\bar o$mura's theorem (see \cite{2}, p.~175) there is a quasi-$m$-accretive operator $A_0$ with the domain $D(A_0)$ dense in $\calk$, which generates the semigroup $S(t)$ on $\calk$, that~is,
\begin{equation}\label{e2.10} 
\frac{d^+}{dt}\,S(t)u_0+A_0 S(t)u_0=0,\ \ff t\ge0,\ u_0\in D(A_0),\end{equation}
and the function $t\to\frac d{dt}\,S(t)$ exists and is everywhere continuous on $(0,\9)$ except a countable set  of $t$. 

We set $\calk=\ov{\calp\cap L^\9}$, i.e., the closure of  $\calp\cap L^\9$ in $H\1$. (As easily seen, we have $\calk=\calp_0\cap H\1$, where $\calp_0$ is the set of all probability measures on  $\rrd$.) Therefore, we have

\begin{corollary}\label{c2.2} The map $u_0\to S(t)u_0$, defined by \eqref{e2.9}, extends to a continuous semigroup of quasi-contractions in $H\1$ on the set $\calk$. Moreover, $S(t)\calk\subset\calk$, $\ff t\ge0$, and $t\to S(t)u_0$ is everywhere differentiable from the right $($in $H\1$-norm$)$ for all $u_0$ in a dense subset $D(A_0)$ of $\calk$. 
\end{corollary}

Taking into account \eqref{e2.3} and  \eqref{e2.10}, we may conclude that
$$A_0u_0=\Delta\beta(u_0)+{\rm div}(Db(u_0)u_0+(K*u_0)u_0),\ \ff u_0\in \calp\cap L^\9\subset D(A_0).$$
For $u_0\in\calk\setminus D(A_0)$, $u^*(t)=S(t)u_0$ is a {\it mild solution} in the sense of \eqref{e1.16}~to 
\begin{equation}\label{e2.11}
\frac{du^*}{dt}+A_0u^*=0,\ \ u^*(0)=u_0.
\end{equation}
We may view the semigroup $S(t)$ as the {\it nonlinear flow} corresponding to NFPE \eqref{e1.1}. As seen above, it leaves invariant $\calk$ and is differentiable on a dense subset of $\calk$.

\section{Proof of Theorem \ref{t2.1}}\label{s3}
\setcounter{equation}{0}

We set, for $\vp>0$,
$$\barr{c}
\beta_\vp(r)\equiv\vp r+\beta((1+\vp\beta)\1r),\ \ff r\in\rr,\vsp
\vf_\vp(r)=\dd\frac1\vp\,\frac r{|r|}\mbox{ if }|r|\ge\frac1\vp,\ \vf_\vp(r)=r\mbox{ if }|r|\le\frac1\vp,\vsp
K_\vp(x)\equiv\eta\(\dd\frac{|x|}\vp\) 
K(x),\ b_\vp(r)\equiv (1-\eta(\vp r))b(r), \earr$$
where $\eta\in C^1([0,\9))$ is a given cut-off function such that $0\le\eta'(r)\le1$, $0\le\eta(r)\le1$, $\ff r\ge0$, and
$$\eta(r)=0,\ \ff r\in[0,1];\ \eta(r)=1,\ \ff r\ge2.$$
We note that by hypothesis (iv) it follows that $b_\vp\in C_b(\rr)\cap C^1_b(\rr)$, $\ff\vp>0$, and
\begin{eqnarray}
&K_\vp\in C^1(\rrd;\rrd),\ \divv\ K_\vp\in L^\9,\ K_\vp\in L^\9(\rrd),\label{e3.1}\\[1mm]
&\dd\lim_{\vp\to0}K_\vp(x)=K(x),\ \ff x\in\rrd. \label{e3.2}
\end{eqnarray}We also set 	$j(r)\equiv\int^r_0\beta(s)ds$ and $j_\vp(r)\equiv\int^t_0\beta_\vp(s)$ and note that $$j_\vp(r)\le j(r),\ \ff r\in\rr.$$
Now, we define the operator $A_\vp:D(A_\vp)\subset H\1\to H\1$,
\begin{equation}\label{e3.3}
\barr{r}
A_\vp(u)=-\Delta\beta_\vp(u)+\divv((D b_\vp(u)u+(K_\vp*\vf_\vp(u))\vf_\vp(u)),\vsp \ff u\in D(A_\vp)=H^1.\earr\end{equation}	
We have

\begin{lemma}\label{l3.1} $A_\vp$ is quasi-m-accretive in $H\1$ and for   $\lbb\in(0,\lbb_\vp)$  
\begin{eqnarray}
&(I+\lbb A_\vp)\1f\ge0,\mbox{ a.e. in }\rrd\mbox{ if }f\ge0,\label{e3.4}\\[1mm]
&(I+\lbb A_\vp)\1(\calp\cap H\1)\subset\calp,\label{e3.5}
\end{eqnarray}Moreover, we have
\begin{eqnarray} 
&0\le (I+\lbb A_\vp)\1 f\le N,\mbox{\ \ a.e. in }\rrd,\label{e3.6}
\end{eqnarray}	
for all $N>0$ and $f\in \calp\cap L^\9$ such that, for $0<\lbb<\frac1\gamma$,\newpage
\begin{eqnarray}
&0\le f\le(1-\lbb\gamma)N,\mbox{ a.e. in }\rrd,\label{e3.7}\\[1mm]
&\gamma=|(\divv\ K)^-|_\9+
\left\|\dd\frac{(K(x)\cdot x)^-}{|x|}\right\|_{L^\9(B_1)}.\label{e3.7a}
\end{eqnarray}	
\end{lemma}

\begin{proof} We set $b^*_\vp(r)=b_\vp(r)r$, $b^*(r)=b(r)r,$ $\ff r\in\rr$, and note that
\begin{equation}\label{e3.8}
b^*_\vp,(b^*_\vp)'\in L^\9(\rr),\ \ff\vp>0.\end{equation}	

\n We have
$$\barr{l}
\<A_\vp(u)-A_\vp(v),u-v\>_{-1}\vsp\qquad
=(\beta_\vp(u)-\beta_\vp(v),u-v)_2
-\<\beta_\vp(u)-\beta_\vp(v),u-v\>_{-1}\vsp\qquad
+\<\divv(D(b^*_\vp(u)-b^*_\vp(v))),u-v\>_{-1}
+\<\divv(K_\vp*\vf_\vp(u-v)),u-v\>_{-1}\vsp\qquad
+\<\divv((K_\vp*\vf_\vp(v))(\vf_\vp(u)-\vf_\vp(v))),u-v\>_{-1}
\ge(\alpha+\vp)|u-v|^2_2\vsp\qquad
-|D(b^*(u)-b^*(v))|_2|u-v|_{-1}
-|K_\vp*\vf_\vp(u-v)|_2|u-v|_{-1}\vsp\qquad
+|K_\vp*\vf_\vp(v)|_2|u-v|_{-1}
\ge(\alpha+\vp)|u-v|^2_2-(|D|_\9|(b^*_\vp)'|_\9\vsp\qquad
+\dd\frac2{\vp^2}\,|K_\vp|_1)|u-v|_2|u-v|_{-1}
=(\alpha+\vp)|u-v|^2_2-C_\vp|u-v|_2|u-v|_{-1},\\\hfill \ff u,v\in H\1.
\earr$$
Hence, $A_\vp$ is $C_\vp$-quasi-dissipative, i.e.,
$$\barr{r}
|(I+\lbb A_\vp)\1 f_1-(I+\lbb A_\vp)\1f_2|_{-1}
\le(1-\lbb\gamma_\vp)\1|f_1-f_2|_{-1},\vsp \ff\lbb\in\(0,\frac1{\gamma_\vp}\),\ f_1,f_2\in\rr(I+\lbb A_\vp),\earr$$
$\gamma_\vp=(4(\alpha+\vp))\1C^2_\vp$. It remains to prove that, for some $\lbb_\vp>0,$ 
$$R(I+\lbb A_\vp)=H^1,\ \ff\lbb\in(0,\lbb_\vp).$$
To this end, we fix $f\in H\1$ and consider the equation $u+\lbb A_\vp(u)=f$, that~is,
\begin{equation}\label{e3.9}
\barr{r}
u-\lbb\Delta\beta_\vp(u)+\lbb\,\divv(D b^*_\vp(u)+(K_\vp*\vf_\vp(u))\vf_\vp(u))=f\vsp
\mbox{ in }\cald'((0,\9)\times\rrd).   \earr\end{equation} 
Equivalently,
\begin{equation}\label{e3.10}
\barr{r}
(I{-}\Delta)\1u{+}\lbb\beta_\vp(u){-}\lbb(I{-}\Delta)\1\beta_\vp(u)
+\lbb(I-\Delta)\1(\divv(Db^*_\vp(u)\vsp\quad
+(K_\vp*\vf_\vp(u))\vf_\vp(u)))=(I-\Delta)\1f.\earr
\end{equation}	
We rewrite \eqref{e3.10} as $G_\vp(u)=(I-\Delta)\1f$, where
$$\barr{ll}
G_\vp(u)\!\!\!&=(I-\Delta)\1u+\lbb\beta_\vp(u)-\lbb(I-\Delta)\1\beta_\vp(u)\vsp 
&+\lbb(I-\Delta)\1(\divv(D_\vp b^*_\vp(u)+(K_\vp*\vf_\vp(u))\vf_\vp(u))),\ \ff u\in L^2,\earr$$
and note that $G_\vp:L^2\to L^2$ is Lipschitz and, as seen above,
$$\barr{l}
(G_\vp(u)-G_\vp(v),u-v)_2\vsp\qquad
\ge|u-v|^2_{-1}+\lbb(\vp+\alpha)|u-v|^2_2-C_\vp\lbb|u-v|_2|u-v|_{-1}\vsp\qquad
\ge(1-(4(\vp+\lbb))\1\lbb C^2_\vp)|u-v|^2,\ \ff u,v\in L^2.\earr$$
Hence, for $\lbb\in(0,\lbb_\vp=4(\vp+\alpha)C^{-2}_\vp),$  the operator $G_\vp$ is monotone and coercive in $L^2$ and, therefore, by the Minty theorem it is surjective, that is, \mbox{$R(G_\vp)=L^2$.} Therefore, equation  \eqref{e3.10} (equivalently, \eqref{e3.9}) has a solution $u_\vp\in D(A_\vp)=H^1$ for each $f\in H\1$, as claimed.
 
Assume now that $f\in L^2$, $f\ge0$, a.e. on $\rrd$ and set $u_\vp=(I+\lbb A_\vp)\1f$, that is, $u_\vp$ is the solution to the equation
\begin{equation}\label{e3.11}
u_\vp-\lbb\Delta\beta_\vp(u_\vp)+\lbb\,\divv(D_\vp b_\vp(u_\vp)u_\vp+(K_\vp*\vf_\vp(u_\vp))\vf_\vp(u_\vp))=f.\end{equation}	
If we multiply \eqref{e3.11} by $\calx_\delta(u^-_\vp)$, where for $\delta>0$,
\begin{equation}\label{e3.12}
\calx_\delta(r)=\left\{
\barr{rcl}
1&\mbox{ if }&r>\delta,\\
\dd\frac1\delta\,r&\mbox{ if }&-\delta\le r\le\delta,\\
-1&\mbox{ if }&r<-\delta,\earr\right.
\end{equation}	
and integrate on $\rrd$, we get
\begin{equation}\label{e3.13a}
\barr{l}
-\dd\int_\rrd u^-_\vp\calx_\delta(u^-_\vp)dx\\
\qquad\ge
\lbb\dd\int_\rrd(D  b_\vp(u_\vp)u_\vp+(K_\vp*\vf_\vp(u_\vp))\vf_\vp(u_\vp))\cdot\nabla u^-_\vp\calx'_\delta(u^-_\vp)dx\vsp 
\qquad+\lbb\dd\int_{[|u^-_\vp|\le\delta]}(|D  b_\vp(u_\vp)|
+|K_\vp*\vf_\vp|u_\vp)|)|\nabla u^-_\vp|dx.\earr\end{equation}
Since $|Db_\vp(u_\vp)|+|K_\vp*\vf_\vp(u_\vp)|\in L^2(Q_T)$ and $\nabla u_\vp\in L^2(Q_T;\rrd)$, it follows that the right hand side of \eqref{e3.13a} goes to zero as $\delta\to0$ and this yields 
$$\lim_{\delta\to0}\int_\rrd u^-_\vp\calx_\delta(u^-_\vp)dx=0.$$
This implies that $u_\vp\ge0$, a.e. in $\rrd.$ By density, this extends to all $f\in H\1$, $f\ge0$, and so \eqref{e3.4} follows. Assume now that $f\in H\1\cap L^1$ and note first that $u_\vp\in L^1$. Since $u_\vp\in H^1$, this follows by multiplying \eqref{e3.11} with $\calx_\delta(u_\vp)$ and integrate on $\rrd$. Then, one obtains as above
\begin{equation}\label{e3.12a}
\limsup_{\delta\to0}\int_\rrd u_\vp\calx_\delta(u_\vp)dx\le\int_\rrd|f|dx,
\end{equation}	 
and, therefore, $|u_\vp|_1\le|f|_1,\ \ff\vp>0$.

If $f\in\calp$, then we have by \eqref{e3.11} that $u_\vp\ge0$ on $Q_T$ and also that
$$\int_\rrd u_\vp dx=\int_\rrd f\,dx=1,\ \ff\vp>0,$$
that is, $u_\vp\in\calp$ and so \eqref{e3.5} follows.

Assume now that $f\in\calp\cap L^\9$, \eqref{e3.7} holds and let's prove \eqref{e3.6}. To this end, we need some sharp estimates on the convolution product $K_\vp*u_\vp$, which will be obtained from hypothesis (iv) by writing $K=K_1+K_2$, where
$$K_1(x)=\calx_{B_1}(x)K(x),\ 
K_2(x)=\calx_{B^c_1}(x)K(x),\ x\in\rrd,$$
where $\calx_{B_1}(x)=1$, $\ff x\in B_1$, 
$\calx_{B_1}(x)=0$, $\ff x\in B^c_1$, and
$\calx_{B^c_1}=1-\calx_{B_1}$. By \eqref{e3.11}, we have
\begin{equation}\label{e3.13}
\barr{l}
\hspace*{-8mm}u_\vp-N-\lbb\Delta(\beta_\vp(u_\vp))
+\lbb\,\divv(Db^*_\vp(u_\vp)-b^*_\vp(N))\vsp
+(K_\vp*\vf_\vp(u_\vp))(\vf_\vp(u_\vp)-\vf_\vp(N))\vsp
=f-N-\lbb b^*_\vp(N)\divv\,D-\lbb\vf_\vp(N)\divv(K_\vp*\vf_\vp(u_\vp))\vsp
=f-N-\lbb b^*_\vp(N)\divv\,D 
-\lbb\vf_\vp(N)
\(\eta\(\dd\frac{|x|}\vp\) 
\divv\,K\)*\vf_\vp(u_\vp)\vsp
-\dd\frac\lbb\vp\,\vf_\vp(N)
\(\(\eta'\(\dd\frac{|x|}\vp\)
\frac{x\cdot K(x)}{|x|}\)*\vf_\vp(u_\vp)\) \vsp 
 -\lbb b^*_\vp(N)\divv\,D\le f-N-\lbb b^*_\vp(N) \divv\,D\vspace*{2mm}\\
 -\dd \frac{\lbb\vf_\vp(N)}\vp\!\!
 \int_{[\vp\le|\bar x|\le2\vp]}
\!\!\eta'\(\frac{|\bar x|}\vp\)\!
\frac{K(\bar x)\cdot\bar x}{|\bar x|} \,\vf_\vp(u_\vp(x{-}\bar x))d\bar x\vspace*{3mm}\\
{+}\lbb\vf_\vp(N)|(\divv\,K)^-|_\9|u_\vp|_1
%\vsp 
\le f{-}N
%{+}\lbb|b|_\9N|(\divv\,D)^-|_\9 
{+}\lbb N
 \left\|\frac{(K(x)\cdot x)^-}{|x|}\right\|_{L^\9(B_1)}
 |u_\vp|_1\vspace*{2mm}\\
  {+}\lbb N|(\divv\,K)^-|_\9|u_\vp|_1
\le f{-}N{+}N\lbb\gamma\le0,\mbox{ a.e. in }\rrd, 
\earr\end{equation}	
where $\gamma$ is defined by \eqref{e3.7}.\newpage\n (We recall that $u_\vp$ and $b^*_\vp$ are nonnegative on $Q_T$ and $|u_\vp|_1\le|f|_1=1$.) 

If multiply \eqref{e3.13} by $\calx_\delta((u_\vp-N)^+)$ and integrate on $\rrd$, we get
$$\barr{l}
\dd\int_\rrd(u_\vp-N)^+\calx_\delta((u_\vp-N)^+)dx
\le-\lbb\int_\rrd\nabla(\beta_\vp(u_\vp)\cdot\nabla\calx_\delta((u_\vp-N)^+)dx\vsp\qquad\dd
+\lbb\int_\rrd D(b^*_\vp(u_\vp)-b^*_\vp(N))\cdot\nabla(\calx_\delta((u_\vp-N)^+)dx\vsp\qquad\dd
+\lbb\int_\rrd(K_\vp*\vf_\vp(u_\vp))(\vf_\vp(u_\vp)-\vf_\vp(N))\cdot\nabla(\calx_\delta((u_\vp-N)^+))dx\vsp\qquad\dd
\le\lbb\int_{[0\le(u_\vp-N)^+\le\delta]}
(|b^{*'}_\vp|_\9|D|+\frac1\vp\,|K_\vp*\vf_\vp(u_\vp)|)
|\nabla(u_\vp-N)^+|dx.
\earr$$because $\beta'_\vp\ge0$, $u_\vp\in H^1$ and  $|\vf_\vp(u_\vp)-\vf_\vp(N)|\le\frac1\vp\,|u_\vp-N|,$  a.e. in $Q_T$.

Taking into account that $|D|$ and $|K_\vp*\vf_\vp(u_\vp)|$ are in $L^2$, we get for $\delta\to0$ that
$$|(u_\vp-N)^+|_1=\lim_{\delta\to0}\int_\rrd(u_\vp-N)^+\calx_\delta((u_\vp-N)^+)dx=0.$$
Hence, $0\le u_\vp\le N$, a.e. in $\rrd$, as claimed.
\end{proof}
 
 Next, we shall take $N=|u_0|_\9$.
 
\begin{lemma}
	\label{l3.2} For each $u_0\in H\1$ there is a unique mild solution $u_\vp\in C([0,\9);H\1)$ to the Cauchy problem
	\begin{equation}\label{e3.14}
	\barr{l}
	\dd\frac{du_\vp}{dt}+A_\vp u_\vp=0\mbox{ on }(0,\9),\vsp
	u_0(0)=u_0,\earr
	\end{equation}
	which is a smooth solution, that is, $\frac{d^+}{dt}u_\vp(t)$ exists everywhere on $[0,\9)$ if $u_0\in H^1=D(A_\vp)$. Moreover, if $u_0\in\calp\cap H\1$, then $u_\vp(t)\in\calp$, $\ff t\ge0,$ 	
	and, for every $T>0$,
\begin{eqnarray}
&u_\vp,\beta_\vp(u_\vp)\in L^2(0,T;H^1)\label{e3.15}\\
&\dd\frac{du_\vp}{dt}\in L^2(0,T;H\1)\label{e3.16}\\
&\divv((K_\vp*u_\vp)\in L^2(0,T;L^2)\label{e3.17}\\[2mm]
&\label{e3.18}
\barr{c}
\dd\frac{du_\vp}{dt}\,(t)-\Delta\beta_\vp(u_\vp(t))+\divv(Db_\vp(u_\vp(t))u_\vp(t)\vsp 
+(K_\vp*\vf_\vp(u_\vp))u_\vp)(t))=0,\mbox{ a.e. }t>0.\earr\\[2mm]
&0\le u_\vp(t,x)\le\exp(\gamma t)|u_0|_\9,\ \mbox{ a.e. }(t,x)\in(0,\9)\times\rrd.\label{e3.19}
\end{eqnarray}
%If $u_0\in L^1\cap H\1$, then
%\begin{equation}\label{e3.21b}
%|u_\vp(t)|_1\le|u_0|_1,\ \ff t\ge0.
%\end{equation}
%Finally, if $u_0\in \calp\cap L^\9$,  then
%\begin{equation}
%0\le u_\vp(t,x)\le\exp(\gamma t)|u_0|_\9,\ \mbox{ a.e. }(t,x)\in(0,\9)\times\rrd.\label{e3.19}
%\end{equation}		
\end{lemma}
 
\begin{proof} Since $A_\vp$ is quasi-$m$-accretive in $H\1$, the existence of a smooth solution $u_\vp\in W^{1,\9}(0,T;H\1)$, $\ff T>0$, to \eqref{e3.14} (respectively, a mild solution for $u_0\in H\1\setminus H^1$) follows by the existence theory recalled for the Cauchy problem \eqref{e1.16} in Section \ref{s1}. Moreover, we have
	$$u_\vp(t)=S_\vp(t)u_0,\ \ff t\ge0,\ u_0\in H\1,$$ where $S_\vp(t)$ is the continuous semigroup in $H\1$ given by the exponential formula \eqref{e1.15}, that is, 
\begin{equation}\label{e3.20}
S_\vp(t)u_0=\lim_{n\to0}\(I+\frac tn\,A_\vp\)^{-n}u_0,\ \ff t\ge0.\end{equation}
By \eqref{e3.4}--\eqref{e3.5} and \eqref{e3.20} it follows that $u_\vp\ge0$, a.,e. in $(0,\9)\times\rrd$, if $u_0\ge0$ and $u_\vp(t)\in\calp$, $\ff t\ge0$, if $u_0\in\calp$. 
 
If $u_0\in H^1=D(A_\vp)$, then $u_\vp$ is a strong solution to \eqref{e3.14}, that is, $t\to u_\vp(t)$ is a.e. differentiable in $H\1$ on $(0,\9)$ and
$$\barr{c}
\dd\frac{d^+}{dt}\,u(t)+A_\vp u_\vp(t)=0,\ \ff t\ge0,\vsp
\dd\frac{du_\vp}{dt}\in L^\9(0,T;H\1),\ u_\vp\in L^\9(0,T;H^1),\ \ff T>0.\earr$$In particular, it follows that \eqref{e3.15}--\eqref{e3.18} hold. Moreover, if multiply \eqref{e3.18} by $u_\vp$ and take into account that
$${}_{H\1}\<\frac d{dt}\,u_\vp(t),u_\vp(t)\>_{H^1}=\frac12\ \frac  d{dt}\,|u_\vp(t)|^2_2,\ \mbox{ a.e. }t>0,$$we get
\begin{equation}\label{e3.21}
\barr{l}
\dd\frac12\,|u_\vp(t)|^2_2+\int^t_0(\nabla\beta_\vp(u_\vp(s)),\nabla u_\vp(s))_2ds\vsp 
 =\dd\frac12\,|u_0|^2_2+
\dd\int^t_0(Db^*_\vp(u_\vp(s))+(K_\vp*\vf_\vp(u_\vp))\vf_\vp(u_\vp(s)),\nabla u_\vp(s))_2ds.\earr\end{equation}
Similarly, by multiplying with $\beta_\vp(u_\vp)$, we get
\begin{equation}\label{e3.22}
 \barr{l}
\dd\int_\rrd  j_\vp(u_\vp(t,x))dx+\int^t_0 |\nabla\beta_\vp(u_\vp(s,x))|^2_2ds=\dd\int_\rrd j(u_0)dx\\
\quad+\dd\int^t_0(Db^*_\vp(u_\vp(s))+(K_\vp*\vf_\vp((u_\vp(s)))\vf_\vp(u_\vp(s)),\nabla\beta_\vp(u_\vp(s)))_2.\earr\end{equation}
By \eqref{e3.21}--\eqref{e3.22} we see that
$$u_\vp\in L^\9(0,T;L^2)\cap L^2(0,T;H^1),\ \ff \vp>0,$$ and so,  by density one can extend \eqref{e3.15} and \eqref{e3.21}--\eqref{e3.22} to all $u_0\in L^2$. 

Finally, if $u_0\in \calp\cap L^\9$,   we have by \eqref{e3.6}--\eqref{e3.7} 
%and \eqref{e3.21b}
	$$0\le\(I+\frac tn\,A_\vp\)\1 u_0\le\(1-\frac tn\,\gamma\)\1|u_0|_\9,\ \ff t>0,\ n\in\nn.$$This yields
	$$0\le\(I+\frac tn\,A_\vp\)^{-n}u_0\le\(1-\frac tn\,\gamma\)^{-n}|u_0|_\9,\ \ff t>0,\ n\in\nn,$$
	and so, by \eqref{e3.20}, for $n\to\9$  we get \eqref{e3.19}, as claimed. %Note also that, if  $u_0\in L^1$, then by \eqref{e3.18} and \eqref{e3.20} it follows that $|u_\vp(t)|_1\le|u_0|_1$.
\end{proof}

\n{\it Proof of Theorem} {\rm\ref{t2.1} (continued).}  Let $u_0\in\calp\cap L^\infty$. Then, by \eqref{e3.21} we~have
$$\barr{l}
|u_\vp(t)|^2_2+\alpha\dd\int^t_0|\nabla u_\vp(s)|^2_2ds\\ \qquad\quad\
\le|u_0|^2_2 +\,\dd\frac12\int^t_0(|Db_\vp(u_\vp(s))|^2_\9+|K_\vp*\vf_\vp(u_\vp(s))|^2_\9)|u_\vp(s)|^2_2ds.\earr$$
We have
$$|Db_\vp(u_\vp(s))|_\9\le|D|_\9|b|_\9,\ \ff s\ge0,$$while by (iv) it follows that
$$\barr{ll}
|K_\vp*\vf_\vp(u_\vp(t))|_\9
\!\!\!&\le\|K\|_{L^1(B_1)}|\vf_\vp(u_\vp(t))|_\9+\|K\|_{L^{p}(B^c_1)}
\|u_\vp(t)\|_{L^{p'}(B^c_1)}\vsp
&\le C(|u_\vp(t)|_\9+|u_\vp(t)|^{\frac1{p'}}_1|u_\vp(t)|^{\frac{p'-1}{p'}}_\9),\ \ff t>0,\earr$$
where $\frac1{p'}=1-\frac1p$. 

Since, as shown earlier, $|u_\vp(t)|_1\le1$ and $|u_\vp(t)|_\9\le C|u_0|_\9$, $\ff t\ge0$, we get the estimate 
\begin{equation}\label{e3.23}
|u_\vp(t)|^2_2+\dd\int^T_0|\nabla u_\vp(t)|^2_2dt\le CT(|u_0|_\9+|u_0|^{\frac1{p'}}_1|u_0|^{\frac{p'-1}{p'}}_\9),\ \ff t\in(0,T).
\end{equation}
Similarly, it follows by \eqref{e3.22} that, for all $T>0$,
\begin{equation}\label{e3.24} 
\int^T_0 |\nabla\beta_\vp(u_\vp(t))|^2_2dt
=C|(u_0)|^2_2+CT(|u_0|_\9+|u_0|^{\frac1{p'}}_1|u_0|^{\frac{p'-1}{p'}}_\9)\le C. \end{equation}
(Here, we have denoted by $C$ several positive constants independent of $\vp$.) Hence, on a subsequence $\{\vp\}\to0$, we have
\begin{equation}\label{e3.25}
\hspace*{-5mm}\barr{rcll}
u_\vp&\to&u^*&\mbox{weakly in }L^2_{\rm loc}(0,\9;H^1)\vsp 
\dd\frac{du_\vp}{dt}&\to&\dd\frac{du^*}{dt}&\mbox{weakly in }L^2_{\rm loc}(0,\9;H\1)\vsp 
u_\vp&\to&u^*&\mbox{weak-star in }L^\9((0,\9)\times\rrd),\\
&&&\mbox{weakly in $L^2_{\rm loc}(0,\9)$, and}\vsp
&&&\mbox{strongly in }L^2_{\rm loc}(0,\9;L^2)\vsp 
Db_\vp(u_\vp)u_\vp&\to&Db(u^*)u^*&\mbox{weakly in }L^1_{\rm loc}(0,\9;L^2)\vsp 
\beta_\vp(u_\vp)&\to&\beta(u^*)&\mbox{weakly in }L^2_{\rm loc}(0,\9;H^1)\vsp 
(K_\vp*\vf_\vp(u_\vp))\vf_\vp(u_\vp)&\to&v^*&\mbox{weak-star in }L^\9((0,\9)\times\rrd).
\earr\end{equation}
Clearly, $u^*$ satisfies the equation
$$\barr{l}
\dd\frac{du^*}{dt}-\Delta\beta(u^*)+\divv(v^*+Db(u^*)u^*)=0\mbox{ in }\cald'(0,\9)\times\rrd\vsp
u^*(0)=u_0.\earr$$
Since, as seen earlier,
$$|K_\vp*\vf_\vp(u_\vp)|_\9\le C,\ \ff\vp>0,$$and by \eqref{e3.1}, \eqref{e3.2}, \eqref{e3.25},
$$K_\vp*\vf_\vp(u_\vp)\vf_\vp(u_\vp)=(K_\vp*u_\vp)u_\vp\to(K*u^*)u^*,$$
a.e. in $(0,\9)\times\rrd$ as $\vp\to0$, we infer that $v^*\equiv(K*u^*)u^*$, and so  \eqref{e2.1}--\eqref{e2.4} hold. More precisely, we have

\begin{eqnarray}
&\barr{l}
\dd\frac{du^*}{dt}-\Delta\beta(u^*)+\divv(Db(u^*)u^*+(K*u^*)u^*)=0\\\hfill\mbox{ in }\cald'((0,\9)\times\rrd)\\
u^*(0,x)=u_0(x),\ x\in\rrd\earr\label{e3.26}\\[3mm]
\!\!\!\!\!\!&u^*,\beta(u^*)\in L^2_{\rm loc}(0,\9;H^1),(K*u^*)u^*\in L^\9((0,\9)\times\rrd)\label{e3.27}\\[1mm]
\!\!\!\!\!\!&u^*(t)\in\calp\cap L^\9,\ \ff t\ge0\label{e3.28}\\[1mm]
\!\!\!\!\!\!&\dd\frac{du^*}{dt}\in L^2(0,T;H\1),\ \ff T>0.\label{e3.29}
\end{eqnarray}
Moreover, since $u^*\in L^\9(0,T;L^\9),$ $\ff T>0$, it follows by \eqref{e3.29} that \eqref{e2.8a} holds.

Let $u^*,v^*$ be two solutions to \eqref{e3.26}--\eqref{e3.29} corresponding to $u_0,v_0\in\calp\cap L^\9$. We have 
\begin{equation}\label{e3.30}
\barr{l}
\!\!\!\!\!\dd\frac12\ \frac  d{dt}\,|u^*(t)-v^*(t)|^2_{-1}
+(\beta(u^*(t))-\beta(v^*(t)),u^*(t)-v^*(t))_2\vsp 
=\<\beta(u^*(t))-\beta(v^*(t)),u^*(t)-v^*(t)\>_{-1}\vsp
-\<\divv(D(b(u^*(t))u^*(t)-b(v^*(t))v^*(t))),u^*(t)-v^*(t)\>_{-1}\vsp
-\<\divv((K*(u^*(t)-v^*(t))u^*(t)),u^*(t)-v^*(t))\>_{-1}\vsp 
-\<\divv((K*v^*(t))(u^*(t)-v^*(t))),u^*(t)-v^*(t)\>_{-1}\vsp
\le(\|\beta'\|_{L^\9(0,N)}+|D|_\9\|(b^*)'\|_{L^\9(0,N)})
|u^*(t)-v^*(t)|_2|u^*(t)-v^*(t)|_{-1}\vsp
+(|K\!*\!(u^*(t)-v^*(t))u^*(t)|_2
{+}|(K\!*\!v^*(t))(u^*(t){-}v^*(t))|_2)
|u^*(t){-}v^*(t)|_{-1},\vsp
\hfill\mbox{ a.e. }t>0. 
\earr\hspace*{-15mm}
\end{equation}
On the other hand, we have 
$$\barr{l}
|(K*(u^*-v^*))u^*|_2\vsp\qquad
\le|((\calx_{B_1}K)*(u^*-v^*))u^*|_2
+|((\calx_{B^c_1}K)*(u^*-v^*))u^*|_2\vsp\qquad
\le\|K\|_{L^1(B_1)}|u^*-v^*|_2|u^*|_\9+|(\calx_{B^c_1}K)*(u^*-v^*)|_r|u^*|_{\frac{2r}{r-2}},\earr$$
where $r\in[2,\9]$.  By the Young inequality and hypothesis (iv)  we have, for $r=\frac{2p}{2-p}$,
$$|(\calx_{B^c_1}K)*(u^*-v^*)|_r
\le\|K\|_{L^p(B^c_1)}|u^*-v^*|_2,$$and, since $|u^*(t)|_1+|u^*(t)|_\9\le C(|u_0|_1+|u_0|_\9),$ this yields
$$|K*(u^*-v^*)u^*|_2\le C|u^*-v^*|_2, \ff t\ge0.$$ Similarly,  it follows that
$$\barr{l}
|(K*v^*)(u^*-v^*)|_2
\le|K*v^*|_\9|u^*-v^*|_2\vsp
\qquad 
\le\|K\|_{L^1(B_1)}|v^*|_\9|u^*-v^*|_2+\|K\|_{L^p(B^c_1)}\|
v^*\|_{L^{p'}(B^c_1)}
|u^*-v^*|_2\vsp
\qquad\le C|u^*-v^*|_2,\mbox{ a.e. }t>0.
\earr$$
Substituting into \eqref{e3.30}, we get  for some $\oo>0$,
$$\frac d{dt}\,|u^*(t)-v^*(t)|_{-1}\le\oo|u^*(t)-v^*(t)|_{-1},\mbox{ a.e. }t\in(0,T),$$and, therefore,
\begin{equation}\label{e3.31}
|u^*(t)-v^*(t)|_{-1}\le\exp(\oo t)|u_0-v_0|_{-1},\ \ff t\ge0.\end{equation}

\begin{remark}\label{r3.3}\rm We note that the first part of hypotheses (iv), that is, $K\in L^1(B_1;\rrd)\cap L^p(B^c_1;\rrd)$ for $p\in[1,2]$ was used in the uniqueness of the solution to \eqref{e1.1} only. So, as for as concerns the existence, hypotheses (i)--(iii) and \eqref{e1.2} are sufficient. It should be noted that \eqref{e1.2} holds the  Biot--Savart kernel \eqref{e1.11b}. 
 \end{remark}

\section{The uniqueness of distributional solutions}\label{s4}
\setcounter{equation}{0}

We shall assume here that {\it besides {\rm(i)--(iv)} the following condition hold}
\begin{itemize}
	\item[\rm(v)] $K\in L^\9(B^c_1;\rrd)$.\end{itemize}
Clearly, (v) implies the last condition in \eqref{e1.2} and is automatically satisfied by all examples of kernels $K$ given in Section \ref{s1}.

We recall that the function $u\in L^1_{\rm loc}([0,T)\times\rrd)$ is said to be a {\it distributional solution} to NFPE \eqref{e1.1} on $Q_T=(0,T)\times\rrd$ if 
\begin{eqnarray}
&\beta(u)\in L^1_{\rm loc}(Q_T)\label{e4.1}\\[1mm]
&(Db(u)+K*u)u\in L^1_{\rm loc}(Q_T;\rrd)\label{e4.2}
\end{eqnarray}
\begin{eqnarray} 
&\barr{r}
\dd\int_{Q_T}(u(t,x)\frac{\pp\vf}{\pp t}\,(t,x)
+\beta(u(t,x))\Delta\vf(t,x)
+u(t,x)(D(x)b(u(t,x))\\
+(K*u(t))(x))\cdot\nabla\vf(t,x))dtdx
+\dd\int_\rrd u_0(x)\vf(0,x)dx=0,\\[3mm]
\ff\vf\in C^\9_0([0,T)\times\rrd).\earr\qquad\label{e4.3}
\end{eqnarray}	
We have
\begin{theorem}\label{t4.1} Let $u_1,u_2\in L^1(Q_T)\cap L^\9(Q_T)$ be two distributional solutions to \eqref{e1.1} such that
\begin{equation}\label{e4.4}
\lim_{t\downarrow0}\mbox{ess\,sup}\int_\rrd(u_1(t,x)-u_2(t,x))\psi(x)dx=0,\ \ff\psi\in C^\9_0(\rrd).\end{equation}	
Then $u_1\equiv u_2.$
\end{theorem}

\begin{proof} We shall use a regularization argument similar to that in \cite{7}, \cite{8} (see, also, \cite{9}). Namely, we set $z=u_1-u_2$, $w=\beta(u_1)-\beta(u_2)$ and get for $z$ the equation
\begin{equation}\label{e4.5}
\barr{r}
\dd\frac{\pp z}{\pp t}-\Delta w+\divv(D((b(u_1)-b(u_2))u_1+b(u_2)z))\vsp
+\divv ((K*z)u_1-z(K*u_2))=0\,\mbox{ in }\cald'(Q_T).\earr\end{equation}	
Consider the operator $\Phi_\vp=(\vp I-\Delta)\1\in L(L^2,L^2).$ We have
\begin{equation}\label{e4.6a}
\vp\Phi_\vp(y)-\Delta\Phi_\vp(y)=y\mbox{ in }\cald'(\rrd),\ \ff y\in L^2.\end{equation}	
We see that $\Phi_\vp\in L(L^2,H^2)$ and applying $\Phi_\vp$ to \eqref{e4.5} we get
\begin{equation}\label{e4.6}
\barr{l}
\dd\frac d{dt}\,\Phi_\vp(z(t))+w(t)-\vp\Phi_\vp(w(t))\vsp\quad
+\divv(\Phi_\vp(z(t))(D((b(u_1(t))-b(u_2(t)))
+b(u_2(t))z(t)))\vsp\quad
+\divv(\Phi_\vp((K*z(t))u_1(t)-z(t)(K*u_2(t))))=0,
\mbox{ a.e. }t\in(0,T).\earr\end{equation}	
(We note that the function $t\to\Phi_\vp(z(t))$ is in $W^{1,2}([0,T];L^2)$, $\ff T>0$, and therefore it is $L^2$-valued absolutely continuous on $[0,T]$.) We set
\begin{equation}\label{e4.7}
h_\vp(t)=(\Phi_\vp(z(t)),z(t))_2=\vp|\Phi_\vp(z(t))|^2_2+|\nabla\Phi_\vp(z(t))|^2_2,\ t\in[0,T],\end{equation}	
and so, by \eqref{e4.6} we see that $h_\vp$ is absolutely continuous on $[0,T]$ and 
\begin{equation}\label{e4.8}
\hspace*{-5mm}\barr{ll}
h'_\vp(t)=2\(\dd\frac d{dt}\,\Phi_\vp(z(t)),z(t)\)_2=-2(w(t),z(t))_2+2\vp(\Phi_\vp(w(t)),z(t))_2\vsp 
+((K*z(t))u_1(t)-z(t)(K*u_2(t)),\nabla\Phi_\vp(z(t)))_2
+(\Psi(t),\nabla\Phi_\vp(z(t)))_2,\vsp\hfill\mbox{ a.e. }t\in(0,T),\earr\end{equation}	
where $\Psi:(0,T)\times\rrd\to\rrd$ is given by
\begin{equation}\label{e4.9}
\Psi\equiv D(b(u_1)-b(u_2)u_1+b(u_2)z)).
\end{equation}	
We have
\begin{equation}\label{e4.10}
h_\vp(0+)=\lim_{t\to0}h_\vp(t)=0,\ \ff\vp>0.\end{equation}	
Indeed, since $t\to\Phi_\vp(z(t))$ is $H^1$-continuous on $[0,T]$, there is \mbox{$\lim\limits_{t\to0}\Phi_\vp(z(t))=g$}  in $H^1$. Now, we may write
$$\barr{ll}
0\!\!\!&
\le h_\vp(t)\le|\Phi_\vp(z(t))-g|_{H^1}|z(t)|_{-1}\vsp&+|g-\vf|_{H^1}|z(t)|_{-1}
+|(z(t),\vf(t))_2|,\ \ff\vf\in C^\9_0(\rrd),\ \ff t\in(0,T).\earr$$
Since $C^\9_0(\rrd)$ is dense in $H^1$, we get by \eqref{e4.4} letting $t\to0$, that \eqref{e4.10} holds.

Next, by (i) and \eqref{e4.8}--\eqref{e4.9} we get
$$\barr{ll}
h'_\vp(t)\!\!\!&+2\alpha|z(t)|^2_2
\le2\vp(\Phi_\vp(w(t)),z(t))_2
+C|z(t)|_2|\nabla\Phi_\vp(z(t))|_2\vsp&+(|(K*z(t))_2u_1(t)|_2+|z(t)(K*u_2(t))|_2)|\nabla\Phi_\vp(z(t))|_2,\mbox{ a.e. }t\in(0,T).\earr$$
By Young's inequality, for $\frac1r=\frac1p-\frac12,$ $\frac1{r'}=1-\frac1r,$ we have by hypothesis~(iv)
$$\barr{ll}
|u_1(K*z)|_2\!\!\!&\le m(B_1)|u_1|_\9\|K*z\|_{L^2(B_1)}
+|u_1|_{r'}\|K*z\|_{L^r(B^c_1)}\vsp 
&\le m(B_1)|u_1|_\9\|K\|_{L^1(B_1)}|z|_2
+|u_1|_{r'}\|K\|_{L^p(B^c_1)}\|z\|_{L^2(B^c_1)}
\le C|z|_2,\earr$$
and by hypothesis (v),
$$|z(K*u_2)|_2\le|z|_2|K*u_2|_\9\le C|z|_2(|u_2|_\9+|u_2|_1)\mbox{ a.e. in }(0,T).$$
This yields
\begin{equation}\label{e4.11}
\barr{r}
h'_\vp(t)+2\alpha|z(t)|^2_2
\le2\vp(\Phi_\vp(w(t)),z(t))_2
+C|z(t)|_2|\nabla\Phi_\vp(z(t))|_2,\vsp\mbox{ a.e. }t\in(0,T),\earr
\end{equation}	
and so, by \eqref{e4.7} we get
$$h'_\vp(t)+\alpha|z(t)|^2_2\le Ch_\vp(t)+2\vp(\Phi_\vp(w(t)),z(t))_2,\mbox{ a.e. }t\in(0,T).$$
Taking into account that $z\in L^\9(Q_T)$ and that $\beta\in C^1(\rr)$, $\beta(0)=0, $  we have
\begin{equation}\label{e4.12}
|w(t)|\le C|z(t)|_2,\ \ff t\in(0,T).\end{equation}
(Here and everywhere in the following we denote by $C$ several positive constants independent of $\vp$.)	
\end{proof}

\bk\n{\bf Claim.} {\it We have, for all $\vp>0$ and some constant $C>0$ independent of $\vp$,}
\begin{equation}\label{e4.13}
\sqrt{\vp}|\Phi_\vp(w(t)),z(t))_2|\le C,\mbox{ a.e }t\in(0,T).\end{equation} 	

\begin{proof} By \eqref{e4.6} we have
\begin{equation}\label{e4.14a}
\vp|\Phi_\vp(w(t))|^2_2+|\nabla\Phi_\vp(w(t))|^2_2=(\Phi_\vp(w(t)),w(t))_2,\ t\in(0,T).\end{equation} 	
	Then, for $d\ge3 $ and $p^*=\frac{2d}{d-2}$ it follows by the Sobolev--Nirenberg--Gagliardo theorem (see, e.g., \cite{9a}, p~278) that
	
	$$(w,\Phi_\vp(w))_2\le|w|_{\frac{2d}{d+2}}|\Phi_\vp(w)|_{p^*}\le C|w|_{\frac{2d}{d+2}}|\nabla\Phi_\vp(w)|_2$$and so, %by \eqref{e4.12} we have
\begin{equation}\label{e4.14aa}
0\le(w(t),\Phi_\vp(w(t)))_2
	\le C|z(t)|_{\frac{2d}{d+2}}
	|\nabla\Phi_\vp(w(t))|_2,\mbox{ a.e. }t\in(0,T),\end{equation} 
	and, similarly,
	\begin{equation}\label{e4.14aaa}
	0\le(z(t),\Phi_\vp(z(t)))_2
	\le C|z(t)|_{\frac{2d}{d+2}}|\nabla\Phi_\vp(z(t))|_2.\end{equation} 
	On the other hand, by \eqref{e4.12} and \eqref{e4.14a} we have
	$$\barr{lcl}
	|\nabla\Phi_\vp(w(t))|^2_2&\le&|\Phi_\vp(w(t))|_2|w(t)|_2
	\le C|z(t)|_2|\Phi_\vp(w(t))|_2\vsp
	&\le&C\vp^{-1}|z(t)|_2|w(t)|_2\le C^2\vp\1|z(t)|^2_2,
	\earr$$
	and so \eqref{e4.14aa} yields
	$$0\le(w(t),\Phi_\vp(w(t)))_2\le C\vp^{-\frac12}|z(t)|_{\frac{2d}{d+2}}|z(t)|_2,\ \mbox{ a.e.   } t\in(0,T).$$
	Similarly, by \eqref{e4.14aaa} we get
	$$(z(t),\Phi_\vp(z(t)))_2\le C\vp^{-\frac12}|z(t)|_{\frac{2d}{d+2}}|z(t)|_2,\ \mbox{ a.e. } t\in(0,T).$$Taking into account that
	\begin{equation}\label{e4.14aaaa}
	|(\Phi_\vp(w(t)),z(t))_2|
	\le(\Phi_\vp(w(t)),w(t))^{\frac12}_2
	(\Phi_\vp(z(t)),z(t))^{\frac12}_2, \end{equation} 
	we get
	$$\barr{ll}
	|(\Phi_\vp(w(t),z(t)))_2|\!\!\!&
	\le C^2\vp^{-\frac12}|z(t)|_{\frac{2d}{d+2}}|z(t)|_2
		\le C^2\vp^{-\frac12}|z(t)|^{\frac{d-2}{2d}}_\9|z(t)|^{\frac{d+2}{2d}}_1|z(t)|_2\vsp
			&\le C^2\vp^{-\frac12}|z(t)|^{\frac{d-1}d}_\9|z(t)|^{\frac{d+1}{d}}_1
			\le C^2\vp^{-\frac12},\ \ff t\in(0,T),\earr$$
			as claimed.

	In the case $d=2$, we have by \eqref{e4.6a}$$\Phi_\vp(w(t))=-E*(w(t)-\vp\Phi_\vp(w(t))),\ \ff t\in(0,T),$$where $E(x)=(\pi)\1\log\frac1{|x|}.$ This yields
\begin{equation}\label{e4.14}
\barr{ll}
|\Phi_\vp(w(t))|_\9\!\!\!
&\le\|E\|_{L^1(B_1)}(|w(t)|_\9+\vp|\Phi_\vp(w(t))|_\9)\vsp
&+\|E\|_{L^\9(B^c_1)}(|w(t)|_1+\vp|\Phi_\vp(w(t))|_1).\earr\end{equation}	
On the other hand, by multiplying \eqref{e4.6a} by $\calx_\delta(\Phi_\vp(w(t))$, where $\calx_\delta$ is given by \eqref{e3.12}, and integrating on  $\rrd$, yields  
$$\vp|\Phi_\vp(w(t))|_1\le|w(t)|_1,\mbox{ a.e. }t\in(0,T),$$
and so, by \eqref{e4.14} we have
$$|\Phi_\vp(w(t))|_\9\le C(|w(t)|_\9+|w(t)|_1\le C(|z(t)|_\9+|z(t)|_1)\le C,\ \ff t\in(0,T).$$
Hence
$$0\le(\Phi_\vp(w(t)),w(t))_2\le C|w(t)|_1\le C,\ \mbox{ a.e.   } t\in(0,T),$$
and, similarly,  
$$0\le(\Phi_\vp(z(t)),z(t))_2\le C(|z(t)|_\9,|z(t)|_1)\le C,\ \mbox{ a.e.   } t\in[0,T].$$
Then,  by \eqref{e4.14aaaa} we get
$$|(\Phi_\vp(w(t)),z(t))_2|\le C,\ \mbox{ a.e.   } t\in[0,T],$$and so \eqref{e4.13} follows. 
  Taking into account \eqref{e4.10} and \eqref{e4.14a}, it follows that
$$h_\vp(t)\le C\int^t_0 h_\vp(s)ds+C\sqrt{\vp},\ \ff t\in[0,T].$$Hence,
$$h_\vp(t)\le C\sqrt{\vp}\exp(Ct),\ \ff t\in[0,T],$$and so $\lim\limits_{\vp\to0}h_\vp(t)=0$ uniformly on $[0,T]$.
 
Recalling \eqref{e4.7}, we get for $\vp\to0$
$$\barr{rcll}
\nabla\Phi_\vp(z)&\to&0&\mbox{ in }L^2(0,T;(L^2)^d)\vsp 
\sqrt{\vp}\,\Phi_\vp(z)&\to&0&\mbox{ in }L^2(0,T;L^2)\vsp 
\Delta\Phi_\vp(z)&\to&0&\mbox{ in }L^2(0,T;H\1)\earr$$
Taking into account that $\vp\Phi_\vp(z)-\Delta\Phi_\vp(z)=z$, we infer that $z=0$ on $(0,T)\times\rrd$, as claimed.\end{proof}

\begin{remark}\label{r4.2}\rm As seen from the above proof, the uniqueness Theorem \ref{t4.1} follows under weaker assumptions that (i)--(iv). For instance, the last conditions in (ii)--(iii) and \eqref{e1.2} are no longer necessary in this case.
\end{remark}

Now, we shall prove for later use a similar uniqueness result for the distributional solution  $v$ to the "freezed" linear version of NFPE \eqref{e4.3}, that~is, 	
\begin{equation}\label{e4.15}
\barr{l}
\dd\frac{\pp v}{\pp t}-\Delta\(\frac{\beta(u)}u\,v\)+\divv((Db(u)+K*u)v)=0\mbox{ in }\cald'(0,T)\times\rrd,\vsp
v(0)=v_0,\earr\end{equation}where $u$ is the solution to \eqref{e1.1} given by Theorem \ref{t2.1}.
In other words,
\begin{equation}\label{e4.16}
\barr{l}
\dd\int_{Q_T}\(\frac{\pp\vf}{\pp t}+\frac{\beta(u)}u\,\Delta\vf+(Db(u)+K*u)\cdot\nabla\vf\)dtdx\vsp\qquad+\dd\int_\rrd v_0(z)\vf(0,x)dx,\ \ff\vf\in C^\9_0([0,T)\times\rrd),\earr\end{equation}	
We have
\begin{theorem}\label{t4.3} Let $u\in L^\9(Q_T)\cap L^1(Q_T)$ be the solution to \eqref{e1.1} and let  $v_1,v_2\in L^1(Q_T)\cap L^\9(Q_T)$ be two distributional solutions to \eqref{e4.5}, which satisfy \eqref{e4.4}. Then, $v_1\equiv v_2$.
\end{theorem}
\begin{proof} Since the proof is identical with that of Theorem \ref{t4.1}, it will be outlined only. Let $z=v_1-v_2.$ We have
	$$\frac{\pp z}{\pp t}-\Delta\(\frac{\beta(u)}u\,z\)+\divv((Db(u)+(K*u))z)=0\mbox{ in }\cald'((0,T)\times\rrd),$$
and, applying the operator $\Phi_\vp$, we get
	$$\barr{r}
	\dd\frac{\pp}{\pp t}\,\Phi_\vp(z)+\frac{\beta(u)}u\,z+\vp\Phi_\vp\(\frac{\beta(u)}u\,z\)
	+\divv(\Phi_\vp((Db(u)+(K*u))z))=0,\\\mbox{a.e. in }(0,T)\times\rrd.\earr$$
	Now, taking into account that $u\in L^\9(Q_T)$, it follows by hypothesis (i) that
	$$0<\alpha\le\frac{\beta(u)}u\le\alpha_1,\ \mbox{ a.e. in }(0,T)\times\rrd,$$ and so we get as above, that	
	  $h_\vp(t)=(\Phi_\vp(z(t)),z(t))_2,$ satisfies the equation
	  
	$$\barr{l}
	h'_\vp(t)+2\alpha|z(t)|^2_2\le2\vp\(\Phi_\vp\(\dd\frac{\beta(u(t))}{u(t)}\,z(t)\),z(t)\)_2\vsp 
\qquad	+(Db(u(t))+(K*u(t)) z(t),\nabla\Phi_\vp(z(t))_2,\mbox{ a.e. }t\in(0,T).\earr$$
Moreover, by \eqref{e4.4} it follows that $h_\vp(0+)=0$ and this yields
	$$ 	h'_\vp(t)+\dd\frac\alpha2\,|z(t)|^2_2
	\le Ch_\vp(t)+2\vp
	\(\Phi_\vp\(\dd\frac{\beta(u(t))}{u(t)}\,z(t)\),z(t)\)_2,\mbox{ a.e. }t\in(0,T),$$
Arguing as in the proof of {\it Claim} \eqref{e4.13}, we get for $d\ge3$
$$\barr{rcl}
\left|\(\Phi_\vp\(\frac{\beta(u(t))}{u(t)}\,
z(t)\),z(t)\)_2\right|
&\le& 
\(\Phi_\vp\(\frac{\beta(u(t))}{u(t)}\,
z(t)\),\frac{\beta(u(t))}{u(t)}\,
z(t)\)^{\frac12}_2\vsp 
((\Phi_\vp(z(t)),z(t))_2)^{\frac12}&\le& C\vp^{-\frac12}(|z(t)|^{\frac{d-1}d}_\9|z(t)|^{\frac{d+1}d}_1)\le C\vp^{-\frac12},\earr$$
and, respectively,
$$\left|\(\Phi_\vp\(\frac{\beta(u(t))}{u(t)}\,z(t)\),z(t)\)_2\right|\le C,\ \ff t\in(0,T),$$if $d=2$.

This implies as in the proof of Theorem \ref{t4.1} that $h_\vp(t)\to0$, $\ff t\in[0,T]$, and so $z\equiv0$, as claimed.
\end{proof}

\section{The existence for the \mkv\\ equation}\label{s5}
\setcounter{equation}{0}

We shall study here the McKean--Vlasov equation \eqref{e1.4} under hypotheses (i)--(v), where $u$ is the solution to NFPE \eqref{e1.1} with $u_0\in\calp\cap L^\9$ given by Theorem \ref{t2.1}. 

As seen in \eqref{e2.8a}, $t\to u(t)$ is narrowly continuous and, therefore, by the superposition principle first applied for the linearized equation \eqref{e4.15} and derived afterwards for the NFPE \eqref{e1.1}, it follows (see \cite{5}, \cite{20} and Theorem~\ref{t5.1} in \cite{8}) the existence  of a probability weak solution $X_t$ to the McKean--Vlasov equation \eqref{e1.4} with the law density $u(t)$. More precisely, we have

{\it There is a stochastic basis $(\ooo,\calf,\mathbb{P})$ with normal filtration $(\calf_t),_{r>0}$ and an $(\calf_t)$-Brownian motion $(W_t,\,t>0)$ such that $X_t$, $t>0$, is an $(\calf_t),_{t>0}$ adapted stochastic process on $(\ooo,\calf,\mathbb{P})$, which satisfies SDE}
\begin{eqnarray}
&&\hspace*{-2mm}dX_t=(D(X_t)b(u(t,X_t)){+}(K{*}u(t,{\cdot}))X_t)dt
{+}\!\(\!\sqrt{\dd\frac{2\beta(u(t,X_t))}{u(t,X_t)}}\!\!\)\!dW_t,\ff t>0,\nonumber\\
&&\hspace*{-2mm}X(0)=X_0,\label{e5.1}\end{eqnarray}and
\begin{equation}
  \call_{X_t}(x)=u(t,x),\ \ff t>0,\ u_0(x)=\call_{X_0}(x)=(\mathbb{P}\circ X^{-1}_0)(x),\ \ff x\in\rrd.  \label{e5.2}
\end{equation}
We may view \eqref{e5.1}--\eqref{e5.2} as a probability representation of solutions $u$ to the NFPE \eqref{e1.1}.
Under hypotheses (i)--(v) we have also from the uniqueness of NFPE \eqref{e4.3} and its linearized version \eqref{e4.15} the weak  uniqueness in law for the solution to the McKean--Vlasov SDE \eqref{e5.1} .  Namely, we~have

\begin{theorem}\label{t5.1}  Let $X_t$ and $\wt X_t$ be   probability weak solutions to \eqref{e5.1} on the stochastic basis $(\ooo,\calf,(\calf_t),\mathbb{P})$, $(\wt\ooo,\wt\calf,(\wt\calf_t),\wt{\mathbb{P}})$ such that  
\begin{eqnarray}
&\call({X_t})=u(t,\cdot),\ \call(\wt X_t)=\wt u(t,\cdot).\label{e5.3}\\[1mm]
&u,\wt u\in L^1(Q)\cap L^\9(Q),\ Q=(0,\9)\times\rrd.\label{e5.4}
\end{eqnarray}
Then $X_t$ and $\wt X_t$ have the same marginal laws, that is,	
\begin{equation}\label{e5.5}
\mathbb{P}\circ X^{-1}_t=\wt{\mathbb{P}}\circ\wt X^{-1}_t,\ \ff t>0.\end{equation}	
\end{theorem}
 
\begin{proof} By It\^o's formula, both functions $u$ and $\wt u$ satisfy \eqref{e4.1}--\eqref{e4.3} and so, by Theorem \ref{t4.1}, $u=\wt u$. Moreover, both marginal laws $\mathbb{P}\circ X^{-1}_t$ and $\wt{\mathbb{P}}\circ \wt X^{-1}_t$ satisfy the martingale problem with the initial condition $u_0$ for the linearized Kolmogorov operator
	$$v\to\Delta\(
	\frac{\beta(u)}{u}\,v\)-(Db(u)+K*u)\cdot\nabla v$$and so, by Theorem 4.3 and the transfer of uniqueness Lemma 2.12 in \cite{21} (see also \cite{22b} and \cite{8}, Sect. 5), it follows that \eqref{e5.5} holds.
\end{proof}

\begin{remark}\label{r5.2}\rm As an immediate consequence of Theorem \ref{t5.1} is also the fact that the family of marginal probability laws $\{t\to\mathbb{P}\circ X^{-1}_t\}_{t\ge0}$, where $X_t$ is the unique probabilistically weak solution to the McKean--Vlasov SDE \eqref{e5.1},  is a nonlinear Markov process on $\calb(\ooo)$ (see Definition 5.1 and Theorem 5.3 in \cite{8}).
\end{remark}

We recall that (see, e.g., \cite{8}, p. 195) that a {\it weak probability solution} $X_t$ to \eqref{e5.1}--\eqref{e5.2} is said to be a {\it strong solution} if it is a measurable function of the Brownian motion $W_t$ or, equivalently, if it is adapted with respect to the completed natural filtration $(\calf^{W_t}_t)_{t>0}$ of $W_t$. It turns out that under appropriate conditions on $b$, $D$ and $K$, the weak probability solution $X_t$ is a strong solution to the McKean--Vlasov equation \eqref{e5.1}--\eqref{e5.2}. Namely, we have

\begin{theorem}\label{t5.3} Assume that, besides hypotheses {\rm(i)--(v)}, the following conditions hold
	\begin{eqnarray}
	&t\to \dd\frac{\beta(r)}r\in C^1_b[0,\9),\ D\in C^1(\rrd;\rrd), \label{e5.6}\\[1mm]
		&\dd\frac\pp{\pp x_i}\,K\in L^2(B^c_1;\rrd),\ i=1,...,d.\label{e5.7}
	\end{eqnarray}	
	Then, the weak probability solution $X_t$ to {\rm\eqref{e5.1}--\eqref{e5.2}} is the unique strong solution to the McKean--Vlasov equation  {\rm\eqref{e5.1}--\eqref{e5.2}}. 
	\end{theorem}

\begin{proof} By the restricted Yamada--Watanabe theorem (see, e.g., \cite{14}, \cite{15}) the conclusion follows if one shows that the stochastic differential equation  \eqref{e5.1} with $u$ given by Theorem \ref{t2.1} has the path  uniqueness property for the weak probability solutions, that is, for every pair of weak solutions $(X,W,(\ooo,\mathbb{P},\calf,(\calf_t)_{t>0})),$ $(\wt X,W,(\ooo,\mathbb{P},\calf,(\calf_t)_{t>0})),$ with $X(0)=\wt X(0)$, we have $\sup|X(t)-\wt X(t)|=0$, $\mathbb{P}$-a.s. In our case, this means the pathwise uniqueness of weak probability solutions to the stochastic differential equation
\begin{equation}\label{e5.8}
\barr{l}
dX_t=f_1(t,X_t)dt+f_2(t,X_t)dW_t\vsp
X(0)=X_0\in\rrd,\earr
\end{equation}in the stochastic basis $(\ooo,\calf,\mathbb{P},W_t)$, where
$$\barr{rcll}
f_1(t,x)&\equiv&D(x)b(u(t,x))+(K*u(t,\cdot))(x),\ t>0,\ x\in\rrd,\vsp
f_2(t,x)&\equiv&\sqrt{\dd\frac{2\beta(u(t,x))}{u(t,x)}},\ t>0,\ x\in\rrd,\earr$$
and $u$ is the solution to the \FP\ equation \eqref{e1.1} given by Theorem \ref{t2.1}. 
The main difficulty for the proof of the path uniqueness for equation \eqref{e5.8} is that the coefficients $f_i$, $i=1,2,$ are not Lipschitz  with respect to the spatial variable $x\in\rrd$ but are only in $H^1$. 
We have indeed for $D_i=\frac\pp{\pp x_i},\ i=1,...,d,$
$$\barr{rcl}
|D_if_2(t,x)|&=&
\left|\(\dd\frac{\beta(u)}u\)'\(\frac{\beta(u)}u\)^{-\frac12}D_iu(t,x)\right| 
 \le  C|D_iu(t,x)|,\\
 &&\hfill \ff(t,x)\in(0,\9)\times\rrd,\earr$$
$$\barr{rcl}
|D_i(D(x)b(u(t,x)))|&\le&|D|_\9|b(u)|_\9+|D|_\9|b'(u)|_\9|\nabla u(t,x)|\vsp
&\le& C|D_iu(t,x)|,\ (t,x)\in(0,\9)\times\rrd,\earr$$
$$\barr{l}
D_i(K*u(t))(x)=\dd\int_{B_1}K(\bar x)D_iu(t,x-\bar x)dx+\dd\int_{B^c_1}D_iK(\bar x)u(t-\bar x)dx\\
\qquad+\dd\int_{[|\bar x|=1]}K(\bar x)u(t,x-\bar x)d\bar x,\  i=1,...,d,\,(t,x)\in(0,\9)\times\rrd. 
\earr$$
Taking into account that, as seen by Theorem \ref{t2.1}, $u\in L^\9(0,T;H^1)$, these formulae imply that
$$\barr{ll}
|D_if_1(t,\cdot)|_2\le\!\!\!& C|D_iu(t,\cdot)|_2+\|K\|_{L^1(B_1)}
|D_iu(t,\cdot)|_2\vsp
&+\|D_iK\|_{L^1(B^c_1)}|u(t,\cdot)_2+C\|K\|_{L^\9(B_1)}|u(t,u(t))|_\9,\vsp&\hfill \ff t\in[0,T].\earr$$
By the above estimations, we see that $f_1,f_2$ have the $H^1$-Sobolev regularity with respect to the spatial variable $x$ and so the pathwise uniqueness of the weak solution $X$ to \eqref{e5.8} follows %(see Theorem 4.4 in \cite{14}) 
by the same argument as that used for the uniqueness of Lagrangian flows ge\-ne\-ra\-ted by nonlinear differential equations in $\rrd$ with coefficients in $L^1(0,T;H^1)$ (see Theorem 2.9 in \cite{13}). (See also \cite{14} for the proof of the pathwise uniqueness of weak solutions to a SDE under similar conditions on coefficients.)
\end{proof}

\n{\bf Acknowledgement.} This work was supported by the DFG through SFB 1283/2 2021-317210226 and by a grant of Ministry oif Research, Innovation and Digitization, CNCS-UEFISCDI project  PN-III-P4-PCE-2021-0006 within PNCDI III. 

\bk\n{\bf Data Availability Statement.} Data sharing is not applicable to this article as no datasets were generated or analyzed during the current study.

\bk\n{\bf Conflict of interest.} {\it The author declares that he has no conflict of interest.}

\end{document}